\documentclass[12pt,leqno,amscd, amsfonts, amssymb,pstricks,verbatim]{amsart}
\usepackage{mathrsfs}
\usepackage{amsmath}
\usepackage{hyperref}
\usepackage{bookmark}
\usepackage{cite}

\def\d{\delta}

\def\a{\alpha}

\def\Z{\mathbb{Z}}
\def\N{\mathbb{N}}

\def\F{\mathbb{F}}

\def\id{\mathbf{id}}

\numberwithin{equation}{section}
\newtheorem{theo}{Theorem}[section]
\newtheorem{defi}[theo]{Definition}
\newtheorem{coro}[theo]{Corollary}
\newtheorem{lemm}[theo]{Lemma}
\newtheorem{prop}[theo]{Proposition}

\newtheorem{rema}[theo]{Remark}

\allowdisplaybreaks

\begin{document}

\title[Biderivations of Lie algebras]{Biderivations of Lie algebras}

\author{Qiufan Chen}

\address{Chen: Department of Mathematics, Shanghai Maritime University,
 Shanghai, 201306, China.}\email{chenqf@shmtu.edu.cn}

\author{Yufeng Yao}

\address{Yao: Department of Mathematics, Shanghai Maritime University,
 Shanghai, 201306, China.}\email{yfyao@shmtu.edu.cn}

\author{Kaiming Zhao}

\address{Zhao: Department of Mathematics, Wilfrid
		Laurier University, Waterloo, ON, Canada N2L 3C5,  and School of
		Mathematical Science, Hebei Normal (Teachers) University, Shijiazhuang, Hebei, 050024 P. R. China.}
	\email{kzhao@wlu.ca}

\subjclass[2010]{17B20, 17B40, 17B65, 17B66}

\keywords{Symmetric biderivation, Witt algebra, finite-dimensional simple Lie algebra, symmetric radical, characteristic subalgebra}

\thanks{This work is supported by National Natural Science Foundation of China (Grant Nos. 12271345 and 12071136) and NSERC (311907-2020). }

\begin{abstract} In this paper, we first introduce the concept of symmetric biderivation radicals and characteristic subalgebras of Lie algebras, and study their properties. Based on these results, we precisely determine biderivations of some Lie algebras including finite-dimensional   simple Lie algebras over arbitrary fields   of characteristic not $2$ or $3$, and the Witt algebras $\mathcal{W}^+_n$ over  fields of characteristic $0$. As an application, commutative post-Lie algebra structure on aforementioned Lie algebras is shown to be trivial.
\end{abstract}

\maketitle
\setcounter{tocdepth}{1}\tableofcontents
\begin{center}
\end{center}

\section{Introduction}


Derivations and generalized derivations are important in the study of structure of various algebras \cite{B1,B2,B3,B4,D,LL,ZF}. Especially, Bre\u{s}ar et al. introduced the notion of biderivation of rings in \cite{B4}, and they showed that all biderivations of noncommutative prime rings are inner. In \cite{W2}, the biderivations of Lie algebras was introduced and the authors proved that all skew-symmetric biderivations of finite dimensional simple Lie algebras over an algebraically closed field of characteristic zero are inner. Furthermore, in \cite{T1}, it was proved that biderivations (without the skew-symmetry restriction) of finite dimensional complex simple Lie algebras are inner. In recent years, many scholars gave their attentions to the study of biderivations of many other Lie (super) algebras using case by case discussion, see \cite{TY,DSL,LGZ,T2,W1,X}. There is no uniform method to determine all biderivations of some classes of Lie (super) algebras.

As is well-known, any biderivation can be decomposed into a sum of a skew-symmetric biderivation and a symmetric biderivation. Skew-symmetric biderivations which are connected with linear commuting map have been deeply studied, such as, all skew-symmetric biderivations on any perfect and centerless Lie algebras are inner biderivations \cite{B5}. Symmetric biderivations can determine commutative post-Lie algebra structures, which are related to the homology of partition posets and Koszul operads \cite{V}. However, up to now, there is no efficient tool to  determine all  symmetric biderivations on   Lie algebras. In the present paper, by introducing the concepts of symmetric biderivation radical and characteristic subalgebra, we establish a simple approach to determine all symmetric biderivations of  finite-dimensional classical simple Lie algebras over a field of characteristic different from $2, 3$, and  Witt algebras $\mathcal{W}^+_n$ over a field of characteristic $0$. It is worthwhile to point out that our method is conceptual avoiding a lot of computations and being distinct from other existing papers. We believe that our method may be used to deal with many other Lie algebras.

The paper is organized as follows. In Section 2, we give some fundamental definitions (symmetric biderivation radical, and characteristic subalgebra  of a Lie algebra $L$) and establish some related properties. In Section 3, we prove that every  symmetric biderivation of a finite-dimensional classical simple Lie algebra over arbitrary fields  of characteristic not $2$ or $3$ is trivial. In Section 4, we show that every  symmetric  biderivation of the Witt algebras $\mathcal{W}^+_n$ over   fields of characteristic $0$ is trivial. Finally, we  determine the
commutative post-Lie algebra structure on aforementioned Lie algebras.
\section{Symmetric biderivation radicals and characteristic subalgebras}
Throughout the paper, we denote by $\Z,\,\Z_+,\,\N$ the sets of integers, nonnegative integers, positive integers, respectively. Let $L$ be a Lie algebra over an arbitrary  field $\F$ in the following, unless otherwise stated.  Denote by ${\rm Aut\,}(L)$ the automorphism  group of $L$.

\begin{defi}\label{defi2.1}
\rm
A bilinear map $\delta:L\times L\to L$ is called a {\em  biderivation} if for any $x,y,z\in L$,
$$\delta([x,y],z)=[x,\delta(y,z)]-[y,\delta(x,z)],\quad \delta(x, [y,z])=[\delta(x,y),z]+[y, \delta(x, z)].$$
If furthermore $\delta(x,y) = \delta(y,x)$ for all $x,y\in L$, we call $\delta$ a {\em symmetric biderivation}.
\end{defi}

\begin{rema}
\begin{itemize}
\item[(1)]
A bilinear map $\delta:L\times L\to L$ is a biderivation if and only if both $\delta(x,\cdot)$ and $\delta(\cdot,x)$ are derivations for every $x\in L$.
\item[(2)]
The set of all biderivations on $L$ is a vector space.
\end{itemize}
\end{rema}
\begin{lemm}\label{prop1}
If $\delta: L\times L\to L$ is a symmetric biderivation, then
$$\delta(x,[y,z]) +\delta(y,[z,x])+\delta(z,[x,y])=0,\ \ \ \ \forall x,y,z\in L.$$
\end{lemm}
\begin{proof} A direct computation gives
\begin{align*}
&\delta(x,[y,z]) +\delta(y,[z,x])+\delta(z,[x,y])\\
=&[y,\d(x,z)]-[z,\d(x,y)]+[z,\d(y,x)]-[x,\d(y,z)]+[x,\d(z,y)]-[y,\d(z,x)]\\
=&0.
\end{align*}
\end{proof}
\begin{lemm}\label{prop2}
Let $L$ be a Lie algebra admitting a Cartan decomposition
$L=\oplus_{\alpha\in \frak{h}^*}L_{\alpha}$ with respect to a fixed Cartan subalgebra $\frak{h}$. Let $x\in L_{\alpha}, y\in L_{\beta}$ with $\alpha\neq\beta\in \frak{h}^*$ and $\delta: L\times L\to L$ be a symmetric biderivation. If $[x,y]=0$, then $\delta(x,y)=0.$
\end{lemm}

\begin{proof}
Since $\alpha\neq \beta$, there exists $h\in \frak{h}$ such that $\alpha(h)\neq \beta(h)$.
By Lemma \ref{prop1}, we have
$$\delta(x,[y,h])+\delta(y,[h,x])+\delta(h,[x,y])=0,$$
which forces $\left(\alpha(h)-\beta(h)\right)\delta(x,y)=0$, that is $\delta(x,y)=0$, as desired.
\end{proof}

\begin{lemm}\label{prop3}
Let $\delta: L\times L\to L$ be a symmetric biderivation. Then for any $\sigma\in {\rm Aut\,}(L)$, $\delta_{\sigma}: L\times L\to L$ is a symmetric biderivation,
where
$$\delta_{\sigma}(x,y)=\sigma(\delta(\sigma^{-1}(x),\sigma^{-1}(y))), \ \ \ \ \forall x,y\in L.$$
\end{lemm}
\begin{proof} We have
$$\delta_{\sigma}(x,y)=\sigma(\delta(\sigma^{-1}(x),\sigma^{-1}(y)))=\sigma(\delta(\sigma^{-1}(y),\sigma^{-1}(x)))=\delta_{\sigma}(y,x),$$
\begin{align*}
\delta_{\sigma}([x,y],z)&=\sigma(\delta([\sigma^{-1}(x),\sigma^{-1}(y)],\sigma^{-1}(z)))\\
&=\sigma\big([\sigma^{-1}(x), \delta(\sigma^{-1}(y),\sigma^{-1}(z))]-[\sigma^{-1}(y), \delta(\sigma^{-1}(x),\sigma^{-1}(z))]\big)\\
&=[x, \sigma(\delta(\sigma^{-1}(y),\sigma^{-1}(z)))]-[y, \sigma(\delta(\sigma^{-1}(x),\sigma^{-1}(z)))]\\
&=[x,\delta_{\sigma}(y,z)]-[y,\delta_{\sigma}(x,z)]),
\end{align*}
and
\begin{align*}
\delta_{\sigma}(x,[y,z])&=\delta_{\sigma}([y,z],x)\\
&=[y,\delta_{\sigma}(z,x)]-[z,\delta_{\sigma}(y,x)]\\
&=[\delta_{\sigma}(x,y),z]+[y,\delta_{\sigma}(x,z)]
\end{align*}
completing the proof.
\end{proof}
\begin{rema}
It follows from Lemma \ref{prop3} that there is a group action of ${\rm Aut\,}(L)$ on the space of symmetric biderivations.
\end{rema}
Now we introduce a very useful tool for the study of symmetric biderivations.
\begin{defi}
\rm
Denote
 $${\rm Rad\,}(L)=\{x\in L\mid \delta(x, L)=0\;\mbox{for any symmetric biderivation}\; \delta: L\times L\to L\},$$
which is called the {\em symmetric biderivation radical} of the Lie algebra $L$.
\end{defi}
Next we assemble a few simple properties about ${\rm Rad\,}(L)$.
\begin{prop}\label{propl}
The following statements hold.
\begin{itemize}
\item[(1)] ${\rm Rad\,}(L)$ is a  subalgebra of $L$;
\item[(2)] If $L$ is a finite dimensional Lie algebra, then ${\rm Rad\,}(L)$ is a closed subset of $L$ with respect to the Zariski topology;
\item[(3)] ${\rm Aut\,}(L)$ stabilizes ${\rm Rad\,}(L)$.
\end{itemize}
\end{prop}
\begin{proof}
$(1)$ and $(2)$ follow from a direct computation.

$(3)$ follows from
$$\delta(\sigma(x), y))=\sigma(\delta_{\sigma^{-1}}(x,\sigma^{-1}(y)))=0,\ \ \ \ \forall x\in {\rm Rad\,}(L), y\in L, \sigma\in {\rm Aut\,}(L).$$
\end{proof}
In order to better investigate the last property of symmetric biderivation radical listed above,  we give the following definitions in Lie algebra case, analogous to those which arise in group theory.
\begin{defi}
\rm
\begin{itemize}
\item[(1)] A subalgebra $K$ of $L$ is called  a {\em characteristic subalgebra}  if $\sigma(K)\subseteq K$ for any $\sigma\in {\rm Aut\,}(L)$;
\item[(2)] An ideal $I$ of $L$ is called  a {\em characteristic ideal}  if $\sigma(I)\subseteq I$ for any $\sigma\in {\rm Aut\,}(L)$;
\item[(3)] $L$ is called {\em characteristically simple} if $L$ has no proper characteristic ideal, that is no characteristic ideal other than $L$ and $0$.
\end{itemize}
\end{defi}
According to the above definitions and Proposition \ref{propl}, we see that ${\rm Rad\,}(L)$ is a characteristic subalgebra of $L$ and any simple Lie algebra is characteristically  simple. In group theory, a finite characteristically  simple group is characterized by a direct sum of some isomorphic simple groups. Moreover, for the Lie algebra case, we also have the following parallel result.
\begin{prop}\label{theo11}
A finite dimensional Lie algebra $L$ is characteristically simple if and only if $L$ is a direct sum of some isomorphic simple Lie algebras.
\end{prop}
\begin{proof}
First we assume that $L=S_1 \oplus S_2 \oplus \cdots \oplus S_k$ and $S_i's$  are isomorphic simple Lie algebras.  If $k=1$, then $L$ is a simple Lie algebra. It is certainly characteristically simple. In the following, suppose $k>1$. For any $i,j\in\{1,\ldots,k\}$ with $i\neq j$, it is apparent that there exists  $\sigma_{i,j}\in {\rm Aut\,}(L)$ such that $$\sigma_{i,j}(S_i)= S_j, \sigma_{i,j}(S_j)= S_i, \sigma_{i,j}(S_l)=S_l, \forall \,\,l\neq i,j.$$ Noticing that $L$ is semisimple and each ideal of $L$ is a sum of certain $S_i$'s. Then for any nonzero proper ideal $I$ of $L$, there exists
some $\sigma_{i,j}\in {\rm Aut\,}(L)$ such that
$\sigma_{i,j}(I) \not\subseteq I$. Consequently, $L$ is characteristically simple.

Conversely, suppose that $L$ is characteristically simple and $S$ is a minimal ideal of $L$. So $S\neq0$ and  it is possible that $S=L$. Let
$$\mathcal{V}=\{N\unlhd L\mid N=S_1 \oplus S_2 \oplus \cdots \oplus S_k, k\in\mathbb{N}, S_i\unlhd L, S_i\cong S, \forall\,1\leq i\leq k\},$$
where $N\unlhd L, S_i\unlhd L$ mean that $N$ and $S_i$ are ideals of $L$. Note that each $S_i$ is a minimal ideal  of $L$. As $S\in \mathcal{V}$,  $\mathcal{V}$ is certainly nonempty. Let $N=S_1\oplus S_2 \oplus \cdots \oplus S_k \in \mathcal{V}$ be of largest possible dimension. We assert that $N=L.$  Otherwise, $N$ is not a characteristic ideal since  $L$ is characteristically simple. There must exists a  $\sigma\in{\rm Aut\,}(L)$ such that $\sigma(N) \not\subseteq N.$  Namely, there exists $i$ such that $\sigma (S_i)  \not\subseteq N.$ Note that  $\sigma (S_i)\cong S$  is a minimal ideal of $L$. Since  $N\cap \sigma (S_i)$ is an ideal of $L$, and $N\cap \sigma (S_i)$ is properly contained in $\sigma (S_i)$, we see that $N\cap \sigma (S_i)=0$  by minimality of $\sigma (S_i)$. So
$$N\oplus\sigma (S_i)= S_1 \oplus S_2 \oplus \cdots \oplus S_k\oplus\sigma (S_i)  $$
is an ideal of $L$, which means that $N\oplus\sigma (S_i)\in  \mathcal{V}$, contradicting the choice of  $N$.
Therefore
$$L=N= S_1 \oplus S_2 \oplus \cdots \oplus S_k.$$
It remains to check that $S$ is simple. We may assume that $S=S_1$. If $I$ is an ideal of $S_1$, then $I$ is an ideal of  $L$. Note that $S_1$ is a minimal ideal of $L$, we see either $I=0$ or $I=S_1$. Hence $S_1$ and   also $S$ is simple, completing the proof.
\end{proof}
\section{Symmetric biderivations of finite-dimensional classical simple Lie algebras}
Let $\mathfrak{g}$ be a finite-dimensional classical simple Lie algebra over a field $\mathbb{F}$ with non-degenerate Killing form and ${\rm char\,} \mathbb{F}\neq 2, 3$ (cf. \cite{Carter,Hum}). If we fix a Cartan subalgebra $\mathfrak{h}$, then $\mathfrak{g}$ has a root space decomposition $\mathfrak{g}=\mathfrak{h}\oplus(\oplus_{\a\in\Phi}\mathfrak{g}_{\a})$, where $\Phi$ is the root system determined by  $\mathfrak{h}$, and $\mathfrak{g}_{\a}$ is the root space corresponding to the root $\a\in\Phi$ with ${\rm dim\,}\mathfrak{g}_{\a}=1$. Denote by $\Phi_{+}$  the set of positive roots, $\Pi$ the set of simple roots,  and $\theta$ the highest root, respectively. Let $\{e_{\alpha}, f_{\alpha}, h_{\beta}\mid \alpha \in \Phi_+,\beta\in\Pi\}$ be a Chevalley basis of $\mathfrak{g}$.

For ${\rm Rad\,}(\mathfrak{g})$, we have the following.
\begin{prop}\label{theo1}
If ${\rm char\,}\mathbb{F}\neq 2, 3$, then ${\rm Rad\,}(\mathfrak{g})\unlhd\mathfrak{g}$.
\end{prop}
\begin{proof}
Keep in mind that ${\rm Rad\,}(\mathfrak{g})$ is a characteristic subalgebra by Proposition \ref{propl}. It suffices to show that ${\rm Rad\,}(\mathfrak{g})$ is also an idea of $\mathfrak{g}$, i.e., we need to show that $[x,r]\in {\rm Rad\,}(\mathfrak{g})$ for any $r\in {\rm Rad\,}(\mathfrak{g})$ and $x\in\mathfrak{g}$. Note that  $\mathfrak{g}$ is generated by  $\mathfrak{g}_{\alpha}(\alpha\in\Phi)$ as a Lie algebra. We just need to prove that $[x_{\a},r]\in {\rm Rad\,}(\mathfrak{g})$ for any $\a\in \Phi$. Recall that any root string is  of length at most $4$. So, for any $\a\in \Phi$, we have $({\rm ad\,}x_{\a})^4=0$. It follows from \cite[Lemma 2.8]{BGP} that
$${\rm exp\,}(\lambda {\rm ad\,}x_{\a})=\id+\lambda\,{\rm ad\,}x_{\a}+\lambda^2\frac{({\rm ad\,}x_{\a})^2}{2}+\lambda^3\frac{({\rm ad\,}x_{\a})^3}{6}\in {\rm Aut\,}(\mathfrak{g}),\ \ \ \ \forall\,\, \lambda \in \mathbb{F}.$$
Applying ${\rm exp\,}(\lambda {\rm ad\,}x_{\a})$ to $r$ and using the fact that ${\rm Rad\,}(\mathfrak{g})$ is a characteristic subalgebra, we get
$$r+\lambda[x_{\a},r]+\lambda^2\frac{({\rm ad\,}x_{\a})^2r}{2}+\lambda^3\frac{({\rm ad\,}x_{\a})^3r}{6}\in {\rm Rad\,}(\mathfrak{g}),\ \ \ \ \forall\,\, \lambda \in \mathbb{F},$$
from which we can obtain a linear equation system whose coefficient matrix is exactly the Vandermonde matrix. Thus $[x_{\a},r]\in{\rm Rad\,}(\mathfrak{g})$ and hence the proposition follows.
\end{proof}

We hope that ${\rm Rad\,}(\mathfrak{g})$ is always an ideal of any Lie algebra $\mathfrak{g}$. However, we are not able to do so in this paper.

We are now in a position to present the following main result, which together with \cite{B5} recovers and generalize the main results in \cite{W2} and \cite{T1} where finite-dimensional   simple Lie algebras over an algebraically closed field $\mathbb{F}$ of characteristic $0$ were considered.

\begin{theo}\label{prop7}
Let $\mathfrak{g}$ be a finite-dimensional classical simple Lie algebra over a field $\mathbb{F}$ with non-degenerate Killing form and ${\rm char\,} \mathbb{F}\neq 2, 3$. Then every symmetric biderivation of $\mathfrak{g}$ is trivial.
 \end{theo}
\begin{proof}
Let $\delta: \mathfrak{g}\times \mathfrak{g}\to \mathfrak{g}$ be a symmetric biderivation. It follows from Lemma \ref{prop1}  that
\begin{equation}\label{equa1}\delta(h_{\alpha}, h)=\delta([e_{\alpha},f_{\alpha}], h)=-2\alpha(h)\delta(e_{\alpha}, f_{\alpha}),\ \ \ \ \forall \a\in\Phi_+, h\in\mathfrak{h}.
\end{equation}
For any fixed $\a\in\Phi_+$, taking $h=h_{\alpha}$ in the above equation gives
\begin{equation}\label{equa2}
\delta(h_{\alpha}, h_{\alpha})=-4\delta(e_{\alpha}, f_{\alpha}),\ \ \ \ \forall \a\in\Phi_+.
\end{equation}
Since $\delta(y,\cdot)$ is a derivation for any $y\in \mathfrak{g}$ and every derivation of $\mathfrak{g}$ is inner (cf. \cite[Theorem 5.3]{Hum}),
it follows that $\delta(h_{\alpha}, h_{\alpha})\in \sum_{\beta\in\Phi}\mathfrak{g}_{\beta}$ for any $\a\in\Phi_+$ and $\delta(e_{\theta}, f_{\theta})\in\frak{h}$. These along with \eqref{equa2} give $\delta(h_{\theta}, h_{\theta})=\delta(e_{\theta}, f_{\theta})=0$. From the fact that each long positive root can become a highest root by the action of Weyl group (cf. \cite[\S\,10.4, Lemma C]{Hum}) and Lemma \ref{prop3}, we see that $\delta(h_{\alpha}, h_{\alpha})=\delta(e_{\alpha}, f_{\alpha})=0$ for any long positive root $\alpha$. Then it follows from \eqref{equa1} that $\delta(h_{\alpha}, h)=0$ for any $h\in\mathfrak{h}$ and long positive root $\alpha$. Since $\Phi$ is spanned by all long positive roots(\cite[Proposition 8.18]{Carter}), we further obtain
$$\delta(\frak{h}, \frak{h})=0,\quad \delta(e_{\alpha}, f_{\alpha})=0,\ \ \ \ \forall\alpha\in \Phi_+.$$
Take $\beta,\gamma\in \Phi_+$ such that $[ce_{\beta}, e_{\gamma}]=e_{\theta}$ for some $c\in \F$. By Lemma \ref{prop1}, we have $\delta(e_{\theta},e_{\theta})=\delta([ce_{\beta}, e_{\gamma}], e_{\theta})=0$. Then for any long positive root $\alpha$, $\delta(e_{\alpha},e_{\alpha})=0$, which in turn forces
\begin{equation}\label{equa22}
\delta(h_{\alpha},e_{\alpha})=\delta([e_{\alpha},f_{\alpha}],e_{\alpha})=0
\end{equation}
by the definition of biderivations. It is well-known that for every $\alpha\in\Phi$,  the set ${\rm Ker\,}\alpha=\{h\in\mathfrak{h}\mid \alpha(h)=0\}$ is a proper subspace of $\mathfrak{h}$. Therefore, one can find a vector $h_0\in\mathfrak{h}\setminus \cup_{\alpha\in\Phi}{\rm Ker\,}\alpha$. For any $h\in \frak{h}$, let $\delta(h, \cdot)=\sum_{\beta}a_{\beta}{\rm ad\,}x_{\beta}+\mathrm{ad}h'$ with $a_{\beta}\in \mathbb{F}, x_{\beta}\in \mathfrak{g}_{\beta}, h'\in\frak{h}$. Since  $\delta(h, h_0)=0$, it follows that $a_{\beta}=0$ for $\beta\in\Phi$. That is, $\delta(h, \cdot)=\mathrm{ad}h'$. Now for any long positive root $\alpha$, from \eqref{equa22}, we see that
$$2\alpha(h')e_{\alpha}=2\delta(h,e_{\alpha})=\delta([h_{\alpha},e_{\alpha}],h)=\delta([h,e_{\alpha}],h_{\alpha})=0,$$
which yields  $h'=0$ by the fact that $\Phi$ is spanned by the long roots. Consequently, $\mathfrak{h}\subseteq{\rm Rad\,}(\mathfrak{g})$. Thanks to Proposition \ref{propl} together with \cite[Corollary 2.1.13]{CM}  or Theorem \ref{theo1} together with the assumption that $\mathfrak{g}$ is a simple Lie algebra, ${\rm Rad\,}(\mathfrak{g})=\mathfrak{g}$, that is, $\delta$ is trivial. We complete the proof.
\end{proof}


\section{Biderivations of $\mathcal{W}^+_n$}
In this section, we assume that $\mathbb{F}$ is a field of characteristic $0$. For $n\in\N$, let $A_n=\mathbb{F}[t_1, t_2,\ldots,t_n]$ be the polynomial algebra and $\mathcal{W}^+_n={\rm Der\,}(A_n)$ be the Witt algebra. For any $\a=(\a_1, \a_2,\ldots,\a_n)\in\Z_+^n$, we denote $t^{\a}=t_{1}^{\a_1}\cdots t_{n}^{\a_n}$  and $|\alpha|=\alpha_1+\cdots+\alpha_n$, respectively.  Set $d_i=\frac{\partial}{\partial t_i}$ and $\mathcal{D}=\mathrm{span}_{\mathbb{F}}\{d_1,d_2,\ldots,d_n\}.$ Then $\mathcal{W}^+_n$ is a free $A_n$-module with basis $d_i, 1\leq i\leq n$, i.e.,  $\mathcal{W}^+_n=\oplus_{i=1}^nA_nd_i$ with the Lie bracket as follows:
$$[fd_i,gd_j]=fd_i(g)d_j-gd_j(f)d_i, \ \ \ f,g\in \mathbb{F}[t_1,\ldots,t_n], i,j=1,\ldots,n.$$
It is known that $\oplus_{i=1}^n\mathbb{F}t_id_i$ is the Cartan subalgebra of  $\mathcal{W}^+_n$.
\begin{lemm}\label{lemma1}
Keep notations as above. Then $\mathcal{D}\subseteq {\rm Rad\,}(\mathcal{W}^+_n).$
\end{lemm}
\begin{proof} Consider first the situation when $n>1$. It is sufficient to show that $d_i\in{\rm Rad\,}(\mathcal{W}^+_n)$ for any $i=1,\ldots,n$.
Observe that ${\rm exp\,}({\rm ad\,} d_i)\in {\rm Aut\,}(\mathcal{W}^+_n)$ for any $i\in \{1,\ldots,n\}$ because ${\rm ad\,}d_i$ is locally nilpotent.  Take any $i\neq j\in\{1,\ldots,n\}$. From  Lemmas \ref{prop2} and \ref{prop3}, for any symmetric biderivation $\delta$ of $\mathcal{W}^+_n$ and $k\in\Z_+$, we have $\delta(t_i^kd_i,t_i^kt_jd_i)=0$
and
\begin{align*}
\delta\big(({\rm exp\,}({\rm ad\,} d_j))(t_i^kd_i),({\rm exp\,}({\rm ad\,} d_j))(t_i^kt_jd_i)\big)=\delta(t_i^kd_i,t_i^kt_jd_i+t_i^kd_i)=0,
\end{align*}
which imply
\begin{equation*}\delta(t_i^kd_i,t_i^kd_i)=0,\ \ \ \ \forall\, k\in\Z_+.\end{equation*}
In particular, we have
\begin{equation}\label{vb2}\delta(d_i,d_i)=0\ \ \ \ \ {\rm and\,}\ \ \ \ \delta(t_id_i,t_id_i)=0.\end{equation}
Since each derivation of $\mathcal{W}^+_n$ is inner \cite{DZ}, we may assume that $\delta(d_i,\cdot)=\mathrm{ad}(\sum_{p=1}^nf_pd_p)$ with $f_1,\ldots,f_n\in A_n$ . Substituting this into the first equality of \eqref{vb2}, we obtain
$$[d_i,\sum_{p=1}^nf_pd_p]=\sum_{p=1}^nd_i(f_p)d_p=0,$$
which forces
\begin{equation}\label{vb4}d_i(f_p)=0, \ \ \ \ \ \forall\, p=1,\ldots,n.\end{equation}
Meanwhile, the second equality of \eqref{vb2} together with  Lemma \ref{prop3} gives
\begin{align*}
\delta\big(({\rm exp\,}({\rm ad\,}  d_i))(t_id_i),({\rm exp\,}({\rm ad\,} d_i))(t_id_i)\big)
=\delta(t_id_i+ d_i,t_id_i+ d_i)=0,
\end{align*}
yielding that  $\delta(d_i,t_id_i)=0$. In addition, we know that  $\delta(d_i,t_qd_i)=0$ provided that $q\neq i$ by Lemma \ref{prop2}. Putting these together gives $\delta(d_i,t_sd_i)=0, s=1,\ldots,n$. Combining this with  the expression of $\delta(d_i,\cdot)$ and  \eqref{vb4}, we get  for any $s=1,\ldots,n$,
$$0=\Big[\sum_{p=1}^nf_pd_p,t_sd_i\Big]=f_sd_i,$$ which in turn forces $f_p=0, p=1,\ldots,n$. As a result, $\delta(d_i,\cdot)=0$.

Assume now that $n=1$.  Denote $\partial_i=t_1^{i+1}d_1$ for any $i\geq -1$ with  $\partial_{-2}=0$. Then we have
 $[\partial_i,\partial_{j}]=(j-i)\partial_{i+j}$ for any $i,j\geq -1$.
Since each derivation of ${\mathcal{W}}_1^+$ is inner \cite{DZ}, we can write
\begin{equation}\label{vb}
\delta(\partial_0,\cdot)=\mathrm{ad}(\sum_{j\geq -1}a_j\partial_j)\quad{\rm and\,}\quad  \delta(\partial_{-1},\cdot)=\mathrm{ad}(\sum_{l\geq -1}b_l\partial_l)  \end{equation}
with $a_j,b_l\in \F$ for $j,l\geq -1$. By Lemma \ref{prop1}, we have
\begin{align*}
(i+1)\delta(\partial_0,\partial_{i-1})&=\delta(\partial_0,[\partial_{-1},\partial_i])\\
&=-\delta(\partial_i,[\partial_{0},\partial_{-1}])-\delta(\partial_{-1},[\partial_{i},\partial_0])\\
&=\delta(\partial_i,\partial_{-1})+i\delta(\partial_{-1},\partial_{i})=(i+1)\delta(\partial_i,\partial_{-1}),
\end{align*}
forcing $\delta(\partial_0,\partial_{i-1})=\delta(\partial_i,\partial_{-1})$ for any $i> -1$. Inserting \eqref{vb} into this equality, we obtain
$$\sum_{j\geq -1}a_j(i-1-j)\partial_{i+j-1}=\sum_{l\geq -1}b_l(i-l)\partial_{l+i},\ \ \ \ \forall i> -1.$$
By observing the coefficients of $\partial_{i-2}$ and $\partial_{2i}$ of both sides, we respectively get
$a_{-1}=0$ and $a_{i+1}=0$ for any $i> -1$. That is,  $\delta(\partial_0,\cdot)=a_0\mathrm{ad}\partial_0$, and then $\delta(\partial_0,\partial_0)=0$. Further, by  Lemmas \ref{prop3}, we have
\begin{align*}
\delta\big(({\rm exp\,}(\lambda{\rm ad\,}\partial_{-1})) (\partial_0),({\rm exp\,}(\lambda{\rm ad\,}\partial_{-1} )) (\partial_0)\big)=\delta(\partial_0+\lambda\partial_{-1},\partial_0+\lambda\partial_{-1})=0,\ \ \ \ \forall \lambda\in\mathbb{F},
\end{align*}
which means $\delta(\partial_0,\partial_{-1})=0$, i.e., $a_0[\partial_0,\partial_{-1}]=-a_0\partial_{-1}=0$. Hence $a_0=0$ and $\partial_0\in {\rm Rad\,}(\mathcal{W}^+_1)$. Finally, for any $\lambda\in\mathbb{F}$ and $ i\geq -1$, we have
\begin{align*}
&\delta\big(({\rm exp\,}(\lambda{\rm ad\,}\partial_{-1})) (\partial_0),({\rm exp\,}(\lambda{\rm ad\,}\partial_{-1})) (\partial_i)\big)\\
&=\delta(\partial_0+\lambda\partial_{-1},\partial_i+\lambda(i+1)\partial_{i-1}+\cdots+\lambda^{i+1}\frac{(i+1)\ldots 1}{(i+1)!}\partial_{-1})\\
&=\delta(\lambda\partial_{-1},\partial_i+\lambda(i+1)\partial_{i-1}+\cdots+\lambda^{i+1}\frac{(i+1)\ldots 1}{(i+1)!}\partial_{-1})=0.
\end{align*}
Taking $\lambda=1,\ldots,i+2$, we can obtain a linear equation system whose coefficient matrix is the Vandermonde matrix. So $\delta(\partial_{-1}, \partial_{i})=0$ and then $d_1\in \mathrm{Rad}(\mathcal{W}^+_1)$. We complete the proof.
\end{proof}
\begin{lemm}\label{lemma2} Let $x\in \mathcal{W}^+_n$.
If $[x, y]\in \mathrm{Rad}(\mathcal{W}^+_n)$ for all $y\in \mathcal{D}$, then $x\in \mathrm{Rad}(\mathcal{W}^+_n)$.
\end{lemm}
\begin{proof}
 For any $y\in \mathcal{D}, z\in \mathcal{W}_n^+$, we have
$$0=\delta([x,y],z)=[\delta(x,z),y]+[x,\delta(y,z)]=[\delta(x,z),y],$$
where the last equation holds by Lemma \ref{lemma1}. This implies that $\delta(x,z)\in \mathcal{D}$.
Assume that $\delta(x,\cdot)=\mathrm{ad}X$ for some $X\in \mathcal{W}^+_n$.
Then $[X,z]\in \mathcal{D}$ for any $z\in\mathcal{W}^+_n$. This yields that $X=0$.
Consequently, $\delta(x,\cdot)=0$, as desired.
\end{proof}
Now we are in a position to present our main result in this section.
\begin{theo}\label{theomp}
Every symmetric biderivation of $\mathcal{W}^+_n$ over any field of characteristic $0$ is trivial.
\end{theo}
\begin{proof}
It is enough to prove that $t^{\alpha}d_i\in \mathrm{Rad}(\mathcal{W}^+_n)$ for any $\alpha\in\Z_+^n$ and $1\leq i\leq n$. For any fixed $i$, we  proceed by induction on $|\alpha|$.
The case $|\alpha|=0$ is given by Lemma \ref{lemma1}. Assume that $t^{\beta}d_i\in \mathrm{Rad}(\mathcal{W}^+_n)$ for any $|\beta|<|\alpha|$.
According to the assumption,  $[d_k,t^{\alpha}d_i]=t^{(\alpha_1,\ldots,\alpha_k-1,\ldots, \alpha_n)}d_i\in \mathrm{Rad}(\mathcal{W}^+_n)$ for any $1\leq k\leq n$. Then by Lemma \ref{lemma2}, we have $t^{\alpha}d_i\in \mathrm{Rad}(\mathcal{W}^+_n)$ for any $\alpha$, completing the proof.
\end{proof}

Combining this with the results in \cite{B5} we obtain the following consequence.

\begin{coro}
	Every  biderivation $\delta$ of $\mathcal{W}^+_n$ over any field $\F$ of characteristic $0$ is  of the form  $\delta(x,y) = \lambda [x,y]$, $x,y\in \mathcal{W}^+_n$, for some
	$\lambda\in  \F $.
\end{coro}

We remark that the  biderivations of the Witt algebras $\mathcal{W}_n$ over Laurent polynomials are proved to be inner in \cite{TY}.

\section{Applications}
Post-Lie algebra structure is an important generalization of
left-symmetric algebra structure, which arised in many areas of algebra and geometry \cite{Bu}. Post-Lie algebras, which are related to homology of partition posets and the study of Koszul operads, have been studied by Vallette \cite{V} and Loday \cite{Lo}. In
addition, post-Lie algebras have been studied in connection with isospectral flows, Yang-Baxter equations, Lie-Butcher Series and Moving Frames \cite{ELMM}. The existence of post-Lie algebra structure on a given
pair of Lie algebras turned out to be very meaningful and quite challenging. The authors in \cite{BD} introduced a special class of post-Lie algebra structures, namely commutative post-Lie algebra.  Using the Levi decompositions,  it was proved that any commutative post-Lie algebra structure on a complex (finite-dimensional)  perfect Lie algebra is trivial \cite{BM}. As an application of the previous results, we shall give commutative post-Lie algebra structures on finite-dimensional   simple Lie algebras over arbitrary fields   of characteristic not $2$ or $3$, and the Witt algebras $\mathcal{W}^+_n$ over  fields of characteristic $0$. Let us recall the following definition of a commutative post-Lie algebra.
\begin{defi}\label{defi2.2}
\rm
A {\em commutative post-Lie algebra structure} on a Lie algebra $L$ over a field $\F$ is an $\F$-bilinear product $x\cdot y$  on $L$  satisfying the following identities:
\begin{align*}
x\cdot y&=y\cdot x,\\
[x,y]\cdot z&=x\cdot (y\cdot z)-y\cdot (x\cdot z),\\
x\cdot [y,z]&=[x\cdot y,z]+[y,x\cdot z],\ \ \ \ \forall x,y,z\in L.
\end{align*}
We also say that $(L,[,],\cdot)$ is a commutative post-Lie algebra.
\end{defi}
There is always the trivial commutative post-Lie algebra structure on $L$, given by $x\cdot y=0$ for all $x,y\in L$. However, in general, it is not obvious whether or not a given Lie algebra admits a non-trivial commutative post-Lie algebra structure. The following lemma
shows the connection between commutative post-Lie algebra structure and symmetric biderivation of a Lie algebra.
\begin{lemm}\label{lemmap}(cf. \cite{T2})
Let $(L,[,],\cdot)$ be a commutative post-Lie algebra. If we define a
bilinear map $\delta:L\times L\to L$ by $\delta(x,y)=x\cdot y$ for all $x,y\in L$, then $\delta$ is a symmetric biderivation of $L$.
\end{lemm}
As a consequence of Theorems \ref{prop7}, \ref{theomp} and Lemma \ref{lemmap}, we now give the main result of this section as follows.
\begin{theo}\label{theomp1}
Let $L$ be the  finite-dimensional  classical simple Lie algebras over arbitrary fields   of characteristic not $2$ or $3$, or the Witt algebras $\mathcal{W}^+_n$ over  fields of characteristic $0$. Then every commutative post-Lie algebra structure on $L$ is trivial.
\end{theo}

\subsection*{Acknowledgements}
The authors would like to thank professors Shujuan Wang and Hengyun Yang for their helpful discussions, and are very grateful to the referee for the helpful suggestions and comments.


\begin{thebibliography}{9999}
\bibitem{BGP} G. Benkart, T. Gregory, A. Premet, The recognition theorem for graded Lie algebras in prime characteristic, {\it Mem. Am. Math. Soc.} {\bf 920}, 1-145 (2009).

\bibitem{B1} D. Benkovi\u{c}, Biderivations of triangular algebras, {\it Linear Algebra Appl.} {\bf 431}, 1587--1602 (2009).

\bibitem{B2} M. Bre\u{s}ar,  On generalized biderivations and related maps, {\it  J. Algebra} {\bf 172}, 764--786  (1995).

\bibitem{B3} M. Bre\u{s}ar,  Near-derivations in Lie algebras, {\it  J. Algebra} {\bf 320}, 3765--3772  (2008).

\bibitem{B4} M. Bre\u{s}ar,  W. S. Martindale, C. R. Miers, Centralizing maps in prime rings with involution, {\it  J. Algebra} {\bf 161}, 342--357  (1993).

\bibitem{B5} M. Bre\u{s}ar, K. Zhao, Biderivations and commuting linear maps on Lie algebras, {\it J. Lie
Theory} {\bf 28}, 885--900 (2018).

\bibitem{Bu} D. Burde, Left-symmetric algebras, or pre-Lie algebras in geometry and physics, {\it Cent. Eur. J. Math.}
{\bf 4(3)}, 323--357 (2006).

\bibitem{BD} D. Burde, K. Dekimpe, Post-Lie algebra structures on pairs of Lie algebras, {\it J. Algebra} {\bf 464},
226--245 (2016).

\bibitem{BM} D. Burde, W. A. Moens, Commutative post-Lie algebra structures on Lie
algebras, {\it J. Algebra} {\bf 467}, 183--201 (2016).


\bibitem{Carter} R. Carter,  Lie algebras of finite and affine type, Cambridge University Press, Cambridge, 2005.



\bibitem{CM} D. H. Collingwood, W. M. McGovern, Nilpotent Orbits in Semisimple Lie Algebras, New York 1993 (Van Nostrand).

\bibitem{DSL} M. Dilxat, S. Gao, D. Liu, Super-biderivations and post-Lie superalgebras on some Lie superalgebras, {\it Acta Math. Sin.}
 {\bf 39(9)}, 1736--1754 (2023).

\bibitem{DZ} D. Z. Dokovic, K. Zhao, Generalized Cartan type $W$ Lie algebras in characteristic zero, {\it J. Algebra},  {\bf 195}, 170--210 (1997).


\bibitem{D} Y. Du, Y. Wang, Biderivations of generalized matrix algebras, {\it Linear Algebra Appl.} {\bf 438(11)}, 4483--4499 (2013).

\bibitem{ELMM} K. Ebrahimi-Fard, A. Lundervold, I. Mencattini, H. Munthe-Kaas, Post-Lie algebras and isospectral
flows, {\it SIGMA Symmetry Integrability Geom. Methods Appl.} {\bf 11}, 093 (2015).

\bibitem{Hum} J. E. Humphreys, Introduction to Lie algebras and representation theory, Springer Science Business Media, Berlin, 1972.

\bibitem{LL} F. Leger, E. Luks, Generalized derivations of Lie algebras, {\it J. Algebra} {\bf 228}, 165--203 (2000).

\bibitem{LGZ} X. Liu, X. Guo, K. Zhao, Biderivations of block Lie algebras, {\it Linear Algebra Appl.} {\bf 538}, 43--55 (2018).

\bibitem{Lo} J. L. Loday, Generalized bialgebras and triples of operads, {\it Ast$\acute{e}$risque} {\bf 320},  116 pp (2008).





\bibitem{T1} X. Tang, Biderivations of finite-dimensional complex simple Lie algebras, {\it Linear Multilinear A.} {\bf 66(2)}, 250--259 (2018).

\bibitem{T2} X. Tang, Biderivations, commuting maps and commutative post-Lie algebra structures
on $W$-algebras,  {\it Commun. Algebra} {\bf 45(12)}, 5252--5261 (2017).

\bibitem{TY} X. Tang,  Y. Yang,  Biderivations of the higher rank Witt algebra without anti-symmetric condition, {\it Open Math.}  {\bf 16}, 447--452 (2018).

\bibitem{V} B. Vallette, Homology of generalized partition posets, {\it  J. Pure Appl. Algebra}, {\bf 208}, 699--725 (2007).


\bibitem{W1} D. Wang, X. Yu, Biderivations and linear commuting maps on the Schr\"{o}dinger-Virasoro
Lie algebra, {\it Commun. Algebra} {\bf 41}, 2166--2173  (2013).

\bibitem{W2} D. Wang, X. Yu, Z. Chen, Biderivations of parabolic subalgebras of simple Lie algebras, {\it Commun. Algebra} {\bf 39}, 4097--4104  (2011).

\bibitem{X} C. Xia, D. Wang, X. Han, Linear super-commuting maps and super-biderivations on the
super-Virasoro algebras, {\it Commun. Algebra} {\bf 44}, 5342--5350 (2016).

\bibitem{ZF} J. Zhang,  S. Feng, H. Li, R. Wu, Generalized biderivations of nest algebras, {\it Linear Algebra Appl.} {\bf 418}, 225--233 (2006).


\end{thebibliography}
\end{document}